\documentclass[12pt]{amsart}
\usepackage{amssymb,latexsym,amsmath,amscd,mathrsfs,yfonts,hyperref}
\usepackage{fullpage}
\usepackage{eucal}

\input xy
\xyoption{all}

\newcommand{\ZZ}{\mathbb Z}

\newcommand{\CC}{\mathbb C}

\newcommand{\NN}{\mathbb N}

\newcommand{\EE}{\mathbb E}
\newcommand{\cpt}{\mathbb K}

\def\bK{\mathbf K}

\def\cO{\mathcal O}

\def\Id{\textup{Id}}

\def\deg{\textup{deg}}

\def\C{\textup{C}}

\def\KKcat{\mathtt{KK}}
\def\Z{\textup{Z}}

\def\colim{\textup{colim}}

\def\id{\mathrm{id}}
\def\Im{\textup{Im}}

\def\End{\textup{End}}

\def\St{\mathtt{Stab}}

\def\Ext{\textup{Ext}}

\def\Hom{\textup{Hom}}

\def\ker{\textup{ker}}

\def\Mod{\mathtt{Mod}}
\def\RMod{\mathtt{RMod}}

\def\op{\textup{op}}

\def\prot{\hat{\otimes}}

\def\Ab{\sf Ab}

\def\sC{\text{$\sigma$-$C^*$}}

\def\NSH{\mathtt{NSH}}

\def\K{\textup{K}}

\def\S{\mathcal{S}}
\def\bZ{\mathbf{Z}}

\def\iNSp{\mathtt{NSp}}
\def\iNS{\mathtt{N}\mathcal{S_*}}

\def\cSf{\mathcal{S}^{\textup{fin}}_*}
\def\Sp{\mathtt{Sp}}

\def\k{\textup{k}}
\def\z{\textup{z}}
\def\calg{\mathtt{CAlg}}
\def\cZ{\mathcal{Z}}

\def\KK{\textup{KK}}
\def\E{\textup{E}}

\def\End{\textup{End}}

\def\cC{\mathcal C}
\def\cN{\mathcal N}
\def\csus{\mathtt{\Sigma_{C^*}}}
\def\1{\bf{1}}

\def\Csep{\mathtt{SC^*}}

\def\iCsep{\text{$\mathtt{SC_\infty^*}$}}

\def\dlim{\varinjlim}
\def\cT{\mathcal T}

\def\cD{\mathcal D}

\def\cA{\mathcal{A}}
\def\cB{\mathcal{B}}

\newcommand{\map}{\rightarrow}
\newcommand{\functor}{\rightarrow}

\def\SHo{\mathtt{\Sigma Ho^{C^*}}}
\def\h{\mathtt{h}}

\def\Ind{\textup{Ind}}
\def\hosc{\mathtt{HoSC^*}}

\def\Oinf{\mathcal{O}_\infty}

\def\one{\mathbf{1}}
\def\cQ{\mathcal{Q}}
\def\pNSp{\mathtt{NSp'}}
\def\ikk{\mathtt{KK_\infty}}
\def\izz{\mathtt{ZZ_\infty}}
\def\kkc{\mathtt{kk^{C^*}}}
\def\uct{\mathtt{KK}^{\mathtt{bt}}_{\mathtt{\infty}}}
\def\czb{\mathtt{ZZ}^{\mathtt{bt}}_{\mathtt{\infty}}}
\def\Sigmat{\Sigma^\infty_{S'}}

\newcommand{\beq}{\begin{eqnarray}}
\newcommand{\beqn}{\begin{eqnarray*}}
\newcommand{\eeq}{\end{eqnarray}}
\newcommand{\eeqn}{\end{eqnarray*}}

\theoremstyle{definition}
\newtheorem{thm}{Theorem}[section]
\theoremstyle{definition}

\newtheorem{lem}[thm]{Lemma}
\newtheorem{prop}[thm]{Proposition}
\newtheorem{cor}[thm]{Corollary}
\newtheorem{ex}[thm]{Example}
\newtheorem{defn}[thm]{Definition}
\newtheorem{rem}[thm]{Remark}

\begin{document}

\title{Colocalizations of noncommutative spectra and bootstrap categories}
\author{Snigdhayan Mahanta}
\email{snigdhayan.mahanta@mathematik.uni-regensburg.de}
\address{Fakult{\"a}t f{\"u}r Mathematik, Universit{\"a}t Regensburg, 93040 Regensburg, Germany.}
\subjclass[2010]{19Dxx, 46L85, 55Nxx, 55P42, 46L35}
\keywords{noncommutative spectra, stable $\infty$-categories, triangulated categories, bootstrap categories, (co)localizations, $C^*$-algebras, bivariant $\K$-theory}
\thanks{This research was supported by the Deutsche Forschungsgemeinschaft (SFB 878 and SFB 1085), ERC through AdG 267079, and the Humboldt Professorship of M. Weiss.}

\begin{abstract}
We construct a compactly generated and closed symmetric monoidal stable $\infty$-category $\pNSp$ and show that $\h\pNSp^\op$ contains the suspension stable homotopy category of separable $C^*$-algebras $\SHo$ constructed by Cuntz--Meyer--Rosenberg as a fully faithful triangulated subcategory. Then we construct two colocalizations of $\pNSp$, namely, $\pNSp[\cpt^{-1}]$ and $\pNSp[\cZ^{-1}]$, both of which are shown to be compactly generated and closed symmetric monoidal. We prove that Kasparov $\KK$-category of separable $C^*$-algebras sits inside the homotopy category of $\ikk :=\pNSp[\cpt^{-1}]^\op$ as a fully faithful triangulated subcategory. Hence $\ikk$ should be viewed as the stable $\infty$-categorical incarnation of Kasparov $\KK$-category for arbitrary pointed noncommutative spaces (including nonseparable $C^*$-algebras). As an application we find that the bootstrap category in $\h\pNSp[\cpt^{-1}]$ admits a completely algebraic description. We also construct a $\K$-theoretic bootstrap category in $\h\ikk$ that extends the construction of the UCT class by Rosenberg--Schochet. Motivated by the algebraization problem we finally analyse a couple of equivalence relations on separable $C^*$-algebras that are introduced via the bootstrap categories in various colocalizations of $\pNSp$.
\end{abstract}

\maketitle

\setcounter{tocdepth}{2}
\tableofcontents

\newpage
\begin{center}
{\bf Introduction}
\end{center}

In \cite{MyNSH} we constructed a compactly generated stable $\infty$-category of noncommutative spectra $\iNSp$ primarily with the intention of proving that the noncommutative stable homotopy category is topological according to the definition in \cite{SchTopTri}. The stable $\infty$-category $\iNSp$ affords an ideal framework for the stable homotopy theory of noncommutative spaces. In \cite{MyNSHLoc} we demonstrated that $\iNSp$ is closed symmetric monoidal, which enabled us to produce smashing colocalizations of $\iNSp$ that generalize bivariant (connective) $\E$-theory category and some variants thereof. One aim of our project is to construct similar interesting stable $\infty$-subcategories of noncommutative spectra (possibly by colocalizations) and ideally give purely algebraic descriptions of their homotopy categories. This is the {\em algebraization problem} that pertains to computational aspects. Concurrently this project also settles the long-standing problem of constructing generalized (co)homology theories on the category of $C^*$-algebras. In fact, thanks to Brown representability in this setup (see Theorem 2.23 of \cite{MyNSH} and Remark \ref{BrownRepr}), noncommutative spectra parametrize all generalized (co)homology theories. The crucial property is the carefully designed compact generation of noncommutative spectra.

Amongst the bivariant homology theories present in the literature $\KK$-theory plays a distinguished role as it has proved to be remarkably effective in tackling various problems in topology and geometry (see, for instance, \cite{KasICM, RosKKK, ConBook, HigICM, YuICM}). The assumption of separability is inherent in Kasparov's original definition of $\KK$-theory \cite{KasKK1,KasKK2}. However, for certain applications to index theory and problems in non-metrizable topology this assumption is too restrictive. Moreover, the construction of the Kasparov (composition) product is a very delicate issue in this setup. Extending the Cuntz picture of $\KK$-theory it is possible to construct a $\KK$-theory $\mathtt{kk^{C^*}}$ for nonseparable $C^*$-algebras (see, for instance, \cite{CunMeyRos}), where the composition product can be established quite easily. In Remark 8.29 of \cite{CunMeyRos} the authors state \begin{quote}Although $\mathtt{kk^{C^*}}$ is defined for inseparable $C^*$-algebras, it does not seem the right generalisation of Kasparov theory to this realm ..... \end{quote}

One motivation of this article is to address this point. After the appearance of \cite{ThomThesis} and \cite{MeyNes} bivariant homology theories of separable $C^*$-algebras have been treated via tensor triangulated categories. Triangulated categories do not offer the full strength of homotopy theoretic techniques. The stable $\infty$-category of noncommutative spectra $\iNSp$ resolves this issue satisfactorily. Hence in this article we construct a stable $\infty$-categorical incarnation of the $\KK$-category that is also able to treat nonseparable $C^*$-algebras. Along the way we prove that the $\KK$-category is topological and construct a generalization of the Rosenberg--Schochet bootstrap category in this setting. The existence of the Kasparov (composition) product also follows effortlessly in our setup. The article is organised as follows:

In Section \ref{SepCstab} we construct a variant of noncommutative spectra (denoted by $\pNSp$). The excisive behaviour of $\h\iNSp^\op$ and the $\KK$-category are not compatible, although the difference can often be ignored. The triangulated category $\h\pNSp^\op$ eliminates this difference entirely. Then we show that $\pNSp$ is a compactly generated stable $\infty$-category that is also closed symmetric monoidal (see Proposition \ref{Cgen} and Theorem \ref{SMNSp}). Theorem \ref{NSp} explains how pointed noncommutative spaces and spectra generalize their commutative counterparts. The $\KK$-category $\mathtt{kk^{C^*}}$ is constructed as a localization of a suspension-stable homotopy category of $C^*$-algebras $\SHo$. We show that the two triangulated categories $\SHo$ and $\h\pNSp^\op$ agree when restricted to separable $C^*$-algebras (see Proposition \ref{Sho} for a precise formulation); however, $\h\pNSp$ is a compactly generated triangulated category in contrast with $(\SHo)^\op$ that does not even admit arbitrary products. Since $\pNSp$ is symmetric monoidal, one can construct smashing colocalizations thereof with respect to coidempotent objects. We mostly focus our attention on two smashing colocalizations, namely, $\pNSp[\cpt^{-1}]$ and $\pNSp[\cZ^{-1}]$. The first colocalization is designed to construct a stable $\infty$-categorical $\KK$-category with good homotopy theoretic properties, whereas the second one is chosen to address the algebraization problem in a tractable setting. In subsection \ref{GenBoot} we discuss the basic construction of the bootstrap category generated by a set of compact objects in a closed symmetric monoidal and compactly generated stable $\infty$-category. The intuitive picture is that the objects in the bootstrap category are precisely the ones that can be constructed by simple operations starting from the chosen set of compact objects as the basic building blocks.

If a noncommutative space $X$ is a coidempotent object in $\iNS$, then so is its {\em suspension spectrum} $\Sigmat(X)$ in $\pNSp$. The coidempotent objects in $\iNS$ include the $C^*$-algebra of compact operators $\cpt$ as well as any strongly self-absorbing $C^*$-algebra (like the Jiang--Su algebra $\cZ$). In Section \ref{KKcoloc} we show that the smashing colocalization $\pNSp[\cpt^{-1}]$ is a compactly generated and closed symmetric monoidal stable $\infty$-category (see Theorem \ref{klim}). Then we show that there is a fully faithful exact functor from Kasparov $\KK$-category into $\h\pNSp[\cpt^{-1}]^\op$ (see Theorem \ref{KK}), which also shows that the $\KK$-category is topological (see Proposition \ref{KKtop}). Thus we denote $\pNSp[\cpt^{-1}]^\op$ by $\ikk$ and it is our proposed bivariant $\K$-theory $\infty$-category for arbitrary pointed noncommutative spaces, whose construction was alluded to in Remark 2.29 of \cite{MyNSH}. The category of pointed noncommutative spaces also accommodates nonseparable $C^*$-algebras. Hence the stable $\infty$-category $\ikk$ produces a bivariant $\K$-theory for nonseparable $C^*$-algebras (see Remark \ref{nonsep}) with (arguably) better formal properties than the counterpart in \cite{CunMeyRos}. The main advantage of our approach is the compact generation of $\h\pNSp[\cpt^{-1}]$. Our method is flexible enough to have a much wider scope of applicability; for instance, it can also be used to construct a bivariant $\K$-theory purely in the algebraic setting (see Remark \ref{AlgKtheory}). A bivariant $\K$-theory for {\em $C^*$-spaces} was constructed using ideas from motivic homotopy theory and model categories by {\O}stv{\ae}r \cite{Ost}; the precise relationship between our construction and that of \cite{Ost} has yet to be clearly understood.

In Section \ref{bkk} we study the bootstrap category $\uct$ in $\ikk$ and show that there is a purely algebraic description of the bootstrap category, i.e., there is an additive equivalence $(\h\uct)^\op\simeq D(\ZZ[u,u^{-1}])$ (see Theorem \ref{algDesc} and Remark \ref{nonAlgebraic}). The bivariant $\K$-theory groups of the noncommutative spaces belonging to $\uct$ satisfy a universal coefficient theorem (UCT) (see Proposition \ref{homuct}). They are computable in terms of $\K$-homology groups. The classical UCT in $\KK$-theory \cite{RosSch} expresses the bivariant $\K$-theory groups in terms of $\K$-theory. The category of separable $C^*$-algebras satisfying this form of UCT can be described as a (co)homological localization of Kasparov $\KK$-category \cite{CunMeyRos}. In subsection \ref{Kthuct} we generalize this idea to construct a $\K$-theoretic bootstrap category in $\ikk$ that truly generalizes the Rosenberg--Schochet UCT category to the setting of pointed noncommutative spaces.

The global structure of the stable $\infty$-category of noncommutative spectra $\pNSp$ appears to be quite difficult, since it contains the stable $\infty$-category of spectra $\Sp$ as a full subcategory. Thus it seems prudent to concentrate on certain subcategories that lie away from $\Sp$. The following diamond diagram of colocalizations of $\pNSp$ is illustrative. Note that we have carefully selected the colocalizations to arrive at the diamond shape; there are numerous other colocalizations arising from coidempotent objects of $\pNSp$ that have been omitted.

\beq \label{diamond}
\xymatrix{
&\pNSp[\cO_2^{-1}]\ar[ld]\ar[rd]\\
\pNSp[\cQ^{-1}] \ar[d] && \pNSp[(\cQ\prot\Oinf)^{-1}] \ar[d] \\
\pNSp[M_{2^\infty}^{-1}] \ar[rd] && \pNSp[\Oinf^{-1}] \ar[ld] \\
&\pNSp[\cZ^{-1}]
}
\eeq Actually we also have a sequence of colocalizations between $\pNSp[(\cQ\prot\Oinf)^{-1}]$ and $ \pNSp[\Oinf^{-1}]$: \beqn\;\pNSp[(\cQ\prot\Oinf)^{-1}]\hookrightarrow\cdots \hookrightarrow\pNSp[(M_{2^\infty}\prot M_{3^\infty}\prot\Oinf)^{-1}]\hookrightarrow \pNSp[(M_{2^\infty}\prot\Oinf)^{-1}]\hookrightarrow\pNSp[\Oinf^{-1}].\eeqn Here $M_{2^\infty}$, $M_{3^\infty}$, $M_{5^\infty}$, etc. are UHF algebras of infinite type and so are their $C^*$-tensor products. They are also examples of strongly self-absorbing $C^*$-algebras \cite{TomWin}. In Section \ref{Zcoloc} we analyse the colocalization of $\pNSp$ with respect to the Jiang--Su algebra $\cZ$, denoted by $\pNSp[\cZ^{-1}]$. In this case we fall short of an algebraic description of the bootstrap category. The hindrance is our lack of understanding of the graded endomorphism ring of the unit object in $\h\pNSp[\cZ^{-1}]$ (see Remark \ref{zalgebraization}). In stable homotopy theory one tries to understand the stable homotopy category via Bousfield localizations with respect to various spectra, since the localized categories are often more tractable and occasionally admit algebraic approximations. It is also important to understand the interrelationship between these localizations. Guided by such considerations we introduce two equivalence relations on $C^*$-algebras (see Definition \ref{Ceq} and Definition \ref{BouEq}) and analyse some examples (see Theorem \ref{ColocEx}). 

\medskip
\noindent
{\bf Notations and conventions:} Throughout this article $\prot$ will denote the maximal $C^*$-tensor product. All $C^*$-algebras are assumed to be separable unless otherwise stated. For any $\infty$-category $\cC$ we denote by $\h\cC$ its homotopy category. In the context of $\infty$-categories a functor (resp. limit or colimit) will implicitly mean an $\infty$-functor (resp. $\infty$-limit or $\infty$-colimit). There is a Yoneda embedding $j:\iCsep^\op\map\iNS$ and a separable $C^*$-algebra $A$ is viewed as a noncommutative space via $j(A)$. In the sequel for brevity we suppress $j$ from the notation. By compact (resp. compactly generated) we shall tacitly mean $\omega$-compact (resp. $\omega$-compactly generated). The triangulated category $\SHo$ stands for the suspension-stable homotopy category of $C^*$-algebras. In the sequel we often denote the full triangulated subcategory spanned by the (de)suspensions of {\em separable} $C^*$-algebras also by $\SHo$ and the difference will be clear from the context. We freely use the notation from the articles \cite{LurToposBook,LurHigAlg,MyNSH,MyNSHLoc}. For the benefit of the reader we enlist some important ones below:

\begin{enumerate}
  \item $\cSf$ = $\infty$-category of finite pointed spaces [Notation 1.4.2.5 of \cite{LurHigAlg}].
  \item $\S_*$ = $\infty$-category of pointed spaces [Notation 1.4.2.5 of \cite{LurHigAlg}]
 \item $\iCsep$ = $\infty$-category of separable $C^*$-algebras [Definition 2.2 of \cite{MyNSH}]
 \item $\iNS$ = $\infty$-category of pointed noncommutative spaces [Definition 2.13 of \cite{MyNSH}]
 \item $\Sp$ = stable $\infty$-category of spectra [Definition 1.4.3.1 of \cite{LurHigAlg}]
 \item $\iNSp$ = stable $\infty$-category of noncommutative spectra [Definition 2.19 of \cite{MyNSH}]
\end{enumerate}

\medskip
\noindent
{\bf Acknowledgements:} The author would like to thank T. Nikolaus for helpful discussions. The author has benefited from the hospitality of Max Planck Institute for Mathematics, Bonn and Hausdorff Research Institute for Mathematics, Bonn under various stages of development of this project. The author is also extremely grateful to the anonymous referee for the constructive report suggesting several improvements.

\section{A generalization of the suspension-stable homotopy category of $\Csep$} \label{SepCstab}
In Section 8.5 of \cite{CunMeyRos} the authors constructed a suspension-stable homotopy category of all $C^*$-algebras denoted by $\SHo$. Although the construction in \cite{CunMeyRos} works for all (possibly nonseparable) $C^*$-algebras, we may (and later on we shall) restrict our attention to separable $C^*$-algebras. The aim in this section is to construct a variant of the stable $\infty$-category of noncommutative spectra, denoted by $\pNSp$, such that $\h\pNSp^\op$ and $\SHo$ agree when restricted to separable $C^*$-algebras. The triangulated category $\h\pNSp$ is actually quite large (it is compactly generated) so that it is able to accommodate nonseparable $C^*$-algebras. Our construction will differ from that of \cite{CunMeyRos} both in methodology and end result for genuinely nonseparable $C^*$-algebras. The triangulated category $\h\pNSp$ has better formal properties (see Remark 3.1 of \cite{MyNSH}). Recall that a presentable symmetric monoidal $\infty$-category is {\em closed} if the tensor product preserves colimits separately in each variable. For the benefit of the reader we record a couple of results from \cite{MyNSHLoc}.

\begin{prop} \label{tensor}
The maximal $C^*$-tensor product on $\Csep$ leads to the following:
\begin{enumerate}
 \item  The $\infty$-categories $\iCsep$ and $\iNS$ are symmetric monoidal. Moreover, the presentable $\infty$-category $\iNS$ is closed symmetric monoidal and the Yoneda functor $j:\iCsep^\op\functor\iNS$ is symmetric monoidal.
 
 \item There is a closed symmetric monoidal stable $\infty$-category $\Sp(\iNS)$, such that the stabilization functor $\Sigma^\infty:\iNS\functor\Sp(\iNS)$ is symmetric monoidal.
\end{enumerate}
\end{prop} Hence there is a symmetric monoidal functor $\St:\iCsep^\op\functor\Sp(\iNS)$ that arises as a composition of two symmetric monoidal functors $\iCsep^\op\overset{j}{\hookrightarrow}\iNS\overset{\Sigma^\infty}{\map}\Sp(\iNS)$. Recall that an extension of $C^*$-algebras $0\map A\map B\map C\map 0$ is called {\em semisplit} if the surjective $*$-homomorphism $B\map C$ admits a completely positive contractive section. Let $\cC$ be any compactly generated stable $\infty$-category and let $V$ be a set of compact objects of $\cC$. Then $\langle V\rangle$ denotes the smallest full stable $\infty$-subcategory of $\cC$ generated the translations (in both directions) and cofibers of the objects of $V$. Consider the collection of morphisms in $\iCsep^\op$ that can be chosen to be a small set $$T'_0=\{ \C(f)\map\ker(f) \,|\, \text{ $0\map \ker(f)\map B\overset{f}{\map} C\map 0$ semisplit extension} \}$$ (see Remark 2.4 of \cite{MyNSH}). We set $$S'_0=\{\St(\theta)\,|\, \theta\in T'_0\},$$ which is also a small set of morphisms in $\Sp(\iNS)$. This defines an exact localization $L_{S'}:\Sp(\iNS)\map S'^{-1}\Sp(\iNS)$ as follows: Set $V=\{\textup{cone}(\theta)\,|\,\theta\in S'_0\}$ and let $\cA=\langle V\rangle$ denote the stable $\infty$-subcategory of $\Sp(\iNS)$ generated by $V$. Note that $\cA$ is a subcategory of the compact objects of $\Sp(\iNS)$. Hence $\Ind_\omega(\cA)$ is a compactly generated stable $\infty$-subcategory of $\Sp(\iNS)$ and we let $S'$ denote the collection of maps in $\Sp(\iNS)$, whose cones lie in $\Ind_\omega(\cA)$. The collection of maps $S'$ defines an accessible localization $L_{S'}:\Sp(\iNS)\functor S'^{-1}\Sp(\iNS)$, which is the desired one (for the details see Section 2.4 of \cite{MyNSH}). By construction there is a short exact sequence of stable presentable $\infty$-categories: $$\Ind_\omega(\cA)\map\Sp(\iNS)\map S'^{-1}\Sp(\iNS).$$

\begin{defn}
Due to the obvious analogy with the $\infty$-category of noncommutative spectra $\iNSp = S^{-1}\Sp(\iNS)$, we denote the $\infty$-category $S'^{-1}\Sp(\iNSp)$ by $\pNSp$. The stable $\infty$-category $\pNSp$ is yet another candidate for noncommutative spectra as we are shortly going to demonstrate (see Theorem \ref{NSp} below).
\end{defn}

\begin{rem}
Evidently, $S'_0\subset S_0$ and $S'\subset S$ from which we obtain a commutative diagram of stable presentable $\infty$-categories (up to equivalence)

\beqn
\xymatrix{
\Sp(\iNS)\ar[rr]^{L_{S'}}\ar[rrd]_{L_S} && S'^{-1}\Sp(\iNS)=\pNSp\ar[d]\\
&& S^{-1}\Sp(\iNS)=\iNSp.
}
\eeqn 
\end{rem}

\begin{thm} \label{SMNSp}
 There is a colimit preserving symmetric monoidal functor $\Sigmat= L_{S'}\circ\Sigma^\infty:\iNS\functor\pNSp$ between presentable closed symmetric monoidal $\infty$-categories.
\end{thm}

\begin{proof}
Thanks to Proposition \ref{tensor} part (2) the functor $\Sigma^\infty:\iNS\functor\Sp(\iNS)$ is a colimit preserving symmetric monoidal functor between presentable closed symmetric monoidal $\infty$-categories. Being an accessible localization the functor $L_{S'}:\Sp(\iNS)\functor {S'}^{-1}\Sp(\iNS)=\pNSp$ is a colimit preserving functor between presentable $\infty$-categories. It remains to show that the functor $L_{S'}$ is also symmetric monoidal between closed symmetric monoidal $\infty$-categories. 

Let $f: X\map Y$ be a local equivalence, i.e., $L_{S'}(f)$ is an equivalence in ${S'}^{-1}\Sp(\iNS)$. By construction this means that the cofiber of $f$ lies in $\Ind_\omega(\cA)$, in other words, there is a cofiber sequence $X\overset{f}{\map}Y\overset{g}{\map}Z$ in $\Sp(\iNS)$ with $Z\in\Ind_\omega(\cA)$.  For any $Q\in\Sp(\iNS)$ the induced diagram $X\otimes Q\overset{f\otimes\id}{\map} Y\otimes Q\overset{g\otimes\id}{\map} Z\otimes Q$ is again a cofiber sequence. If we can show that $Z\otimes Q\in\Ind_\omega(\cA)$ then it would imply that $f\otimes\id:X\otimes Q\map Y\otimes Q$ is also a local equivalence. Hence by Proposition 2.2.1.9 and Example 2.2.1.7 of \cite{LurHigAlg} (see also Lemma 3.4 of \cite{GepGroNik}) this would prove that the localization is symmetric monoidal (or compatible with it). It would follow that $\pNSp$ is closed symmetric monoidal (see Remark 3.5 of \cite{GepGroNik}) and we shall have completed the proof.
 
To this end write $Z=\colim_\alpha Z_\alpha$ with $Z_\alpha\in\cA$ and set $Y_\alpha = Y\times_Z Z_\alpha$ and consider the map of cofiber sequences 

\beqn
\xymatrix{X_\alpha \ar[r]^{f_\alpha}\ar[d] & Y_\alpha \ar[r]^{g_\alpha} \ar[d] & Z_\alpha \ar[d]\\
X\ar[r]^f & Y \ar[r]^g & Z,}
\eeqn where $f_\alpha:X_\alpha\map Y_\alpha$ is the fiber of $g_\alpha$. If $V$ is a set of compact objects in a compactly generated stable $\infty$-category $\cC$, then $\langle V\rangle$ denotes the smallest stable $\infty$-subcatgory of $\cC$ generated by $V$. Observe that $\Sp(\iNS)$ is compactly generated by the objects of $\langle\St(\iCsep^\op)\rangle$. Now $Y_\alpha$ need not be compact but we may write $Y_\alpha = \colim_\beta Y_{\alpha\beta}$ with each $Y_{\alpha\beta}\in\langle\St(\iCsep^\op)\rangle$. Set $X_{\alpha\beta} = X_\alpha \times_{Y_\alpha} Y_{\alpha\beta}$ and we obtain a ladder diagram of cofiber sequences 

\beqn
\xymatrix{
X_{\alpha\beta}\ar[r]^{f_{\alpha\beta}}\ar[d] & Y_{\alpha\beta}\ar[d]\ar[r]^{g_{\alpha\beta}} & Z_{\alpha\beta} \ar[d] \\
X_\alpha \ar[r]^{f_\alpha}\ar[d] & Y_\alpha \ar[r]^{g_\alpha} \ar[d] & Z_\alpha \ar[d]\\
X\ar[r]^f & Y \ar[r]^g & Z,}
\eeqn The top left square is by construction a pullback square and since we are in a stable $\infty$-category it is also a pushout square. Thus we have equivalences $Z_{\alpha\beta}\overset{\sim}{\map} Z_\alpha$ and $X_\alpha\overset{\sim}{\map} X$. 

Let $\Sp(\iNS)^c$ denote the stable $\infty$-category of compact object in $\Sp(\iNS)$. We thus have a cofiber sequence \beq \label{Cseq} X_{\alpha\beta}\overset{f_{\alpha\beta}}{\map} Y_{\alpha\beta}\overset{g_{\alpha\beta}}{\map} Z_{\alpha\beta}\eeq in $\Sp(\iNS)^c$ with $Z_{\alpha\beta}\in {\cA}$. Moreover, $$\colim_{\alpha\beta}X_{\alpha\beta}\map\colim_{\alpha\beta} Y_{\alpha\beta}\map\colim_{\alpha\beta} Z_{\alpha\beta}$$ is equivalent to the cofiber sequence $X\overset{f}{\map} Y\overset{g}{\map} Z$. Now we write $Q=\colim_\gamma Q_\gamma$ with each $Q_\gamma$ compact in $\Sp(\iNS)$. Using the fact that $\otimes$ commutes with colimits we find that the cofiber sequence $X\otimes Q\overset{f\otimes\id}{\map} Y\otimes Q\overset{g\otimes\id}{\map} Z\otimes Q$ is equivalent to \beq \label{Cseq2} \colim_{\alpha\beta\gamma} (X_{\alpha\beta}\otimes Q_\gamma)\map\colim_{\alpha\beta\gamma} (Y_{\alpha\beta}\otimes Q_\gamma)\map\colim_{\alpha\beta\gamma} (Z_{\alpha\beta}\otimes Q_\gamma).\eeq Our aim is to show that $(Z_{\alpha\beta}\otimes Q_\gamma)\in\cA$. From Proposition 2.17 of \cite{MyNSH} we know that there is a fully faithful exact functor $\St=\Pi^\op:\hosc[\Sigma^{-1}]^\op\map\h\Sp(\iNS)$, whose image lies inside $\h\Sp(\iNS)^c$. The functor $\St$ is also symmetric monoidal with respect to $\prot$ on $\hosc[\Sigma^{-1}]$. The cofiber sequence \eqref{Cseq} gives rise to an exact triangle in the triangulated category $\h\Sp(\iNS)^c$; by construction we may assume that it is of the form $X_{\alpha\beta}\map \St(B,m)\map\St(C,n)\map \Sigma X_{\alpha\beta}$ with $(B,m),(C,n)\in\hosc[\Sigma^{-1}]^\op$. Using the fully faithful exact functor $\St$ we may view the exact triangle to be $ (B,m)\overset{h}{\leftarrow} (C,n)\leftarrow \C(h)\leftarrow \Sigma (B,m)$ with $\St(\C(h))\simeq X_{\alpha\beta}$ in $\h\Sp(\iNS)$. There is also an exact triangle associated with the cofiber sequence $X_{\alpha\beta}\otimes Q_\gamma\map Y_{\alpha\beta}\otimes Q_\gamma\map Z_{\alpha\beta}\otimes Q_\gamma$. If $Q_\gamma =\St(A,k)$, then using the exactness of $\prot$ in the triangulated category $\hosc[\Sigma^{-1}]$ we may write this exact triangle as $$(B,m)\prot (A,k)\overset{h\prot\id}{\leftarrow} (C,n)\prot (A,k)\leftarrow \C(h)\prot (A,k)\leftarrow \Sigma (B,m)\prot (A,k)$$ with $\St(C,n)\prot (A,k))\simeq \St(C,n))\otimes\St(A,k) \simeq Z_{\alpha\beta}\otimes Q_\gamma$. We know that $(C,n)$ belongs to the triangulated subcategory $\mathtt{T}$ of $\hosc[\Sigma^{-1}]$ generated by $\{\textup{cone}(\theta)\,|\, \theta\in T_0\}$ so that $\St(\mathtt{T}) \simeq \h\cA$. Using the fact that $-\prot (A,k)$ is an exact functor on $\hosc[\Sigma^{-1}]$ it is clear $(C,n)\prot (A,k)$ belongs to $\mathtt{T}$ as well. Thus $\St(C,n)\prot (A,k))\simeq Z_{\alpha\beta}\otimes Q_\gamma$ belongs to $\cA$ and hence $\colim_{\alpha\beta\gamma} Z_{\alpha\beta}\otimes Q_\gamma\in \Ind_\omega(\cA)$, i.e., $f\otimes\id: X\otimes Q\map Y\otimes Q$ is a local equivalence. 
\end{proof}

\begin{rem}
Let $T'$ (resp. $T$) denote the strongly saturated collections of morphisms in $\iNS$ generated by $j(T'_0)$ (resp. $j(T_0)$), where $$T_0=\{ \C(f)\map\ker(f) \,|\, \text{ $0\map \ker(f)\map B\overset{f}{\map} C\map 0$ any extension} \}.$$ One can also construct an $\infty$-category of noncommutative spaces, which comes equipped with a canonical functor $T'^{-1}\iNS\functor T^{-1}\iNS$. Moreover, the suspension spectrum functor $\Sigmat$ factors as $\iNS\functor T'^{-1}\iNS\functor\pNSp$.
\end{rem}

\begin{prop} \label{Cgen}
 The stable $\infty$-categories $\iNSp$ and $\pNSp$ are compactly generated.
\end{prop}

\begin{proof}
It is shown in Lemma 2.22 of \cite{MyNSH} that $\iNSp$ is compactly generated. The proof of the compact generation of $\pNSp$ is similar. Indeed, by construction $\iNS$ is compactly generated and hence so is its stabilization $\Sp(\iNS)$ (see Propostion 1.4.3.7 of \cite{LurHigAlg}). Therefore, the accessible localization $L_{S'}:\Sp(\iNS)\functor\pNSp$ is also compactly generated. Note that $S'$ is a strongly saturated collection of morphisms generated by a small set, such that the domain and the codomain of every morphism in this small set is compact. One can also argue as follows: by construction there is a short exact sequence of stable presentable $\infty$-categories $\Ind_\omega(\cA)\map \Sp(\iNS)\map\pNSp$, which induces a short exact sequence of triangulated categories $\h\Ind_\omega(\cA)\map \h\Sp(\iNS)\map\h\pNSp$. We know that $\h\Sp(\iNS)$ is compactly generated and since the objects of $\cA$ are compact in $\Sp(\iNS)$ so is $\h\Ind_\omega(\cA)$. From Theorem 7.2.1 (2) of \cite{KraLoc} we deduce that $\h\pNSp$ is a compactly generated triangulated category. Since $\pNSp$ is a stable $\infty$-category, it must itself be compactly generated (see Remark 1.4.4.3 of \cite{LurHigAlg}).
\end{proof}

\begin{rem} \label{BrownRepr}
 Using arguments similar to Theorem 2.23 and Remark 2.25 of \cite{MyNSH} one can show that both $\h\pNSp$ and $(\h\pNSp)^\op$ satisfy Brown representability.
\end{rem}

Now we compare our construction with $\SHo$ (restricting our attention to separable $C^*$-algebras). Recall that the objects in $\SHo$ are pairs $(A,n)$ with $A\in\Csep$ and $n\in\ZZ$. Its morphisms are defined as $$\SHo( (A,n),(B,m)):= {\dlim}_k\, [ J_{cpc}^{n+k} A, \Sigma^{m+k} B],$$ where the functor $J_{cpc} A$ is defined by the short exact sequence $0\map J_{cpc} A\map T_{cpc} A \map A\map 0$ (see Section 8.5 of \cite{CunMeyRos} for the details). There is a composite functor $\iCsep^\op\overset{j}{\map}\iNS\overset{\Sigma^\infty}{\map} \Sp(\iNS),$ whose opposite functor is denoted by $\Pi:\iCsep\map \Sp(\iNS)^\op$ and another composite functor $\iCsep^\op\overset{j}{\map}\iNS\overset{\Sigmat}{\map} \pNSp,$ whose opposite functor is denoted by $\pi:\iCsep\map \pNSp^\op$.

\begin{prop} \label{Sho}
There is a fully faithful functor $\Theta:\SHo\functor\h \pNSp^\op$ induced by the functor $\pi:\iCsep\map \pNSp^\op$; in particular, for $A,B\in\Csep$ there is a natural isomorphism $$\SHo(A,B)\cong\h \pNSp^\op(\pi(A),\pi(B)).$$
\end{prop}

\begin{proof}
The homotopy category $\h\Sp(\iNS)^\op$ is triangulated. The canonical composite functor $\Pi:\h\iCsep=\hosc\map\h\iNS^\op\map\h\Sp(\iNS)^\op$ inverts the suspension functor $\Sigma$. Thus it factors through $\hosc[\Sigma^{-1}]$, i.e., we have the following commutative diagram:
\beqn
\xymatrix{
\h\iCsep\ar[rr]^\Pi\ar[rd]_\iota && \h\Sp(\iNS)^\op \\
& \hosc[\Sigma^{-1}]\ar@{-->}[ru]_{{\Pi}}.
}
\eeqn The dashed functor ${\Pi}$ is fully faithful (see Proposition 2.17 of \cite{MyNSH}), i.e., one has $$\hosc[\Sigma^{-1}](\iota(A),\iota(B))\cong{\dlim}_k [\Sigma^k A,\Sigma^k B]\cong\h\Sp(\iNS)^\op(\Pi(A),\Pi(B)).$$

Thanks to the universal characterization of $\SHo$ (see Section 8.5 of \cite{CunMeyRos}) one can obtain it as a Verdier quotient $\hosc[\Sigma^{-1}]\overset{V}{\map}\SHo$ of triangulated categories with respect to the set of maps $\iota(T'_0)$. We have an exact colocalization $L_{S'}^\op: \h\Sp(\iNS)^\op\map\h\Sp(\iNS)^\op$, whose essential image $\h \pNSp^\op$ is spanned by the $S'$-colocal objects. Now $\ker(V)$ is the thick subcategory of $\hosc[\Sigma^{-1}]$ generated by $\{\textup{cone}(f)\,|\, f\in\iota(T'_0)\}$ and $\ker(L_{S'}^\op)$ is the colocalizing subcategory of $\h\Sp(\iNS)^\op$ generated by ${\Pi}(\ker(V))$. Thus we obtain a unique functor $\Theta:\SHo\functor\h \pNSp^\op$ making the following diagram commute:

\beqn
\xymatrix{
&&\ker(V)\ar[d]\ar[r]^{{\Pi}} & \ker(L_{S'}^\op) \ar[d] \\
\h\iCsep\ar[rr]^\iota\ar[rrd] &&\hosc[\Sigma^{-1}]\ar[d]_V \ar[r]^{{\Pi}} & \h\Sp(\iNS)^\op \ar[d]^{L_{S'}^\op}\\
&&\SHo \ar@{-->}[r]^{\Theta} & \h \pNSp^\op.
}
\eeqn Observe that the composite functor $L_{S'}^\op\circ {\Pi}\circ\iota$ is $\pi$ and $V$ acts as identity on objects whence $\Theta$ acts as $\pi$ on objects. Taking the opposite of the above diagram one can argue as in Theorem 2.26 of \cite{MyNSH} to show that $\Theta$ is fully faithful. The argument relies on Neeman's generalization of Thomason localization theorem \cite{NeeBook} (see also \cite{KraLoc} for the formulation used in \cite{MyNSH}).
\end{proof}

\begin{thm} \label{NSp}
Let $\mathcal{S_*}$ denote the $\infty$-category of pointed spaces and $\Sp$ denote the stable $\infty$-category of spectra. We have the following: \begin{enumerate}
\item There is a fully faithful $\omega$-continuous functor $\mathcal{S_*}\hookrightarrow\iNS$.
\item If $\cC$ denotes either of the two compactly generated stable $\infty$-categories $\iNSp$, $\pNSp$, then there is a fully faithful colimit preserving exact functor $\Sp\hookrightarrow\cC$.
 \end{enumerate}
\end{thm}

\begin{proof}
For the first assertion notice that the Gel'fand--Na{\u{\i}}mark correspondence gives a fully faithful functor $\mathcal{S_*^{\mathtt{fin}}}\hookrightarrow\iCsep^\op$. Hence there is a fully faithful $\omega$-continuous functor $\mathcal{S_*}=\Ind_\omega(\mathcal{S_*^{\mathtt{fin}}})\hookrightarrow \Ind_\omega(\iCsep^\op) =\iNS$ (see Proposition 5.3.5.11 (1) of \cite{LurToposBook}).
 
 Let us prove that there is a fully faithful colimit preserving exact functor $\Sp\hookrightarrow\iNSp$. Let $S$ denote the sphere spectrum that generates the stable $\infty$-category of finite spectra $\Sp^\mathtt{fin}$ under translations (in both directions) and cofibers in $\Sp$. Sending $S$ to $\pi^\op(\CC)\in\iNSp$ sets up an exact functor $\Sp^\mathtt{fin}\functor\iNSp$, whose image lies inside the compact objects of $\iNSp$. By Theorem 2.26 of \cite{MyNSH} this functor is fully faithful at the level of homotopy categories and hence it is fully faithful. Once again by Proposition 5.3.5.11 (1) of \cite{LurToposBook} it extends to a fully faithful $\omega$-continuous exact functor $\Sp=\Ind_\omega(\Sp^\mathtt{fin})\functor\iNSp$. Thus it preserves small coproducts and hence by Proposition 1.4.4.1 (2) of \cite{LurHigAlg} all colimits. The proof of the corresponding assertion for $\pNSp$ is similar using Proposition \ref{Sho} and hence it is left to the reader. 
 
 The referee has kindly pointed out that in conjunction with Theorem 2.26 of \cite{MyNSH} above we are using that the Gel'fand--Na{\u{\i}}mark correspondence induces a fully faithful functor $\h\Sp^\mathtt{fin}\hookrightarrow\NSH^\op$. This can be seen as follows: for finite pointed CW complexes $(X,x)$ and $(Y,y)$ one has [$\C(X,x)$ = the $C^*$-algebra of continuous functions on $X$ that vanish at $x$] $$\h\Sp^\mathtt{fin}((X,x),(Y,y))\cong {\dlim}_n [\Sigma^n (X,x),\Sigma^n (Y,y)]\cong {\dlim}_n [\C(\Sigma^n(Y,y)),\C(\Sigma^n(X,x))]$$ and $\NSH(\C(Y,y),\C(X,x))= {\dlim}_n [[\Sigma^n\C(Y,y),\Sigma^n\C(X,x)]]\cong {\dlim}_n [[\C(\Sigma^n(Y,y)),\C(\Sigma^n(X,x))]]$, where $[[-,?]]$ denotes asymptotic homotopy classes of asymptotic homomorphisms. Finally, it follows from Corollary 17 of \cite{DadAsymHom} that the canonical homomorphism $$[\C(\Sigma^n(Y,y)),\C(\Sigma^n(X,x))]\map [[\C(\Sigma^n(Y,y)),\C(\Sigma^n(X,x))]]$$ is an isomorphism.
\end{proof}

\begin{rem}
 Let us consider the category $\SHo$ for separable $C^*$-algebras. From Theorem \ref{SMNSp} and Proposition \ref{Sho} we get a symmetric monoidal functor $\pi:\hosc=\h\iCsep\functor\h\pNSp^\op$. The functor $\pi:\h\iCsep\map\h\pNSp^\op$ uniquely determines $\Theta$ below \beqn
 \xymatrix{
 \h\iCsep\ar[r]^\pi\ar[d]_\iota & \h\pNSp^\op.\\
 \SHo\ar@{-->}[ru]_\Theta
 }\eeqn The category $\SHo$ is actually a tensor triangulated category by setting $(A,m)\otimes (B,n) = (A\prot B, m+n)$, i.e., the tensor structure is compatible with that on $\h\iCsep$ induced by the maximal $C^*$-tensor product so that $\iota$ is symmetric monoidal. The explicit nature of the tensor structure on $\SHo$ can now be used to verify that $\Theta$ must be symmetric monoidal. 
\end{rem}

\subsection{Colocalizations of $\pNSp$} \label{ColocNSp}
Computations in $\pNSp$ are presumably as hard as those in the stable $\infty$-category of spectra $\Sp$. Therefore, we try to understand colocalizations of $\pNSp$ with respect to certain coidempotent objects (see Definition 3.1 of \cite{MyNSHLoc}) that lie away from $\Sp$. The following Lemma follows easily from Theorem \ref{SMNSp}. \begin{lem}
 Let $A$ be a coidempotent object in $\iNS$. Then $\Sigmat(A)$ is a coidempotent object in $\pNSp$. 
\end{lem} If $A=\cpt$ or $A=\cD$, where $\cD$ is any strongly self-absorbing $C^*$-algebra, then $j(A)$ is a coidempotent object in $\iNS$ (see Lemma 3.2 of \cite{MyNSHLoc}). In this case $\Sigmat(A)$ is a coidempotent object in $\pNSp$ and using the dual of Proposition 4.8.2.4 of \cite{LurHigAlg} one concludes that the functor $R_{\Sigmat(A)}:\pNSp\map \pNSp$ sending $X\mapsto X\otimes \Sigmat(A)$ is a colocalization. We simplify the notation by setting $R_A = R_{\Sigmat(A)}$ and denote the essential image $R_A (\pNSp)$ by $\pNSp[A^{-1}]$. For any $A\in\Csep$ we set $A_\cpt = A\prot\cpt$, i.e., the $\cpt$-stabilization of $A$.

\begin{rem} \label{CAlg}
Let $\cC$ be any symmetric monoidal $\infty$-category and let $\calg(\cC)$ denote the $\infty$-category of commutative algebra objects in $\cC$ (see Definition 2.1.3.1 of \cite{LurHigAlg}). Let $E$ be an idempotent object of $\cC$. Then the localization $L_E = -\otimes E :\cC\map\cC$ with $L_E\one_\cC \simeq E$ endows $L_E\cC$ with the structure of a symmetric monoidal $\infty$-category. Moreover, the right adjoint $L_E\cC\hookrightarrow \cC$ induces a fully faithful functor $\calg(L_E\cC)\hookrightarrow\calg(\cC)$ (see Proposition 4.8.2.9 of \cite{LurHigAlg}). Hence $E$ itself is a commutative algebra object in $\cC$. This observation in the special case $\cC=\pNSp^\op$ will play an important role in the sequel.
\end{rem}

\begin{prop} \label{locCompGen}
Let $A\in\Csep$ be such that $\Sigmat(A)$ is a coidempotent object in $\pNSp$. Then the stable $\infty$-category $\pNSp[A^{-1}]$ is compactly generated and closed symmetric monoidal.
\end{prop}

\begin{proof}
It follows from Lemma 3.4 (and Remark 3.5) of \cite{GepGroNik} that $\pNSp[A^{-1}]$ is a closed symmetric monoidal $\infty$-category. It is clear that the $\infty$-category $\pNSp[A^{-1}]$ is stable. Thus it suffices to show that its triangulated homotopy category $\h\pNSp[A^{-1}]$ is compactly generated (see Remark 1.4.4.3 of \cite{LurHigAlg}). Proposition \ref{Cgen} above shows that $\pNSp$ is compactly generated whence so is its triangulated homotopy category. We observe that $R_A:\h \pNSp\map \h\pNSp$ is a coproduct preserving colocalization of triangulated categories, whose essential image is $\h\pNSp[A^{-1}]$. From the functorial triangle $R_A(X) \map X\map L_A(X) \map \Sigma R_A(X)$ we deduce that the corresponding localization $L_A:\h \pNSp\map \h \pNSp$ also preserves coproducts. Hence $\Im(L_A) \simeq \ker(R_A)$ is a compactly generated triangulated category (see Remark 5.5.2 of \cite{KraLoc}). Observe that $\h \pNSp/\ker(R_A)\simeq\Im(R_A) = \h\pNSp[A^{-1}]$ whence $\h\pNSp[A^{-1}]$ is compactly generated (see Theorem 5.6.1 of \cite{KraLoc}). 
 \end{proof}

\begin{prop} \label{SScoloc}
 Let $\cD,\cD'$ be separable $C^*$-algebras such that $\Sigmat(\cD)$ and $\Sigmat(\cD')$ are both coidempotent objects in $\pNSp$. Moreover, let $\iota:\cD\map\cD'$ be a $*$-homomorphism, such that $\cD\prot\cD'\overset{\iota\otimes\id_{\cD'}}{\map}\cD'\prot\cD'$ is homotopic to an isomorphism. Then there is a colocalization $\theta:\pNSp[\cD^{-1}]\functor\pNSp[\cD'^{-1}]$ given by $\theta(-) = -\otimes\Sigmat(\cD\otimes\cD')$, such that $R_{\cD'}\simeq \theta\circ R_\cD$. 
\end{prop}

\begin{proof}
Since $\Sigmat(\cD')\cong\Sigmat(\cD')\otimes\Sigmat(\cD')$ it follows that $$\Sigmat(\cD')\cong\Sigmat(\cD')\otimes\Sigmat(\cD')\map\Sigmat(\cD)\otimes\Sigmat(\cD')\simeq\Sigmat(\cD\otimes\cD')$$ is homotopic to an equivalence. Since $R_\cD$ (resp. $R_{\cD'}$) is $-\otimes\Sigmat(\cD)$ (resp. $-\otimes\Sigmat(\cD')$) and $\Sigmat(\cD)\otimes\Sigmat(\cD')\simeq \Sigmat (\cD\otimes\cD')$, the assertion follows.
\end{proof}

\begin{ex} \label{SScolocEx}
We present two pertinent examples of the above scenario.
\begin{enumerate}
\item If $\cD\map\cD'$ is a unital embedding between strongly self-absorbing $C^*$-algebras, we deduce from the Proposition on page 4027 of \cite{TomWin} that $\cD\prot\cD'\map\cD'\prot\cD'$ is homotopic to an isomorphism. 

\item \label{colocEx} For any strongly self-absorbing $C^*$-algebra $\cD$ the corner embedding $\cD\overset{\iota}{\map}\cD\prot\cpt$ has the property that $\iota\otimes\id_{\cD\prot\cpt}$ is homotopic to an isomorphism (see Proposition 2.9 of \cite{MyNSHLoc}). Note that $\cpt$ itself is not a strongly self-absorbing $C^*$-algebra.
\end{enumerate}
\end{ex}

\begin{rem}
Observe that the functor $\pNSp[\cD'^{-1}]\map\pNSp[\cD^{-1}]$, which is the left adjoint to $\theta$ in the above Proposition \ref{SScoloc}, is a colimit preserving fully faithful functor between compactly generated stable $\infty$-categories.
\end{rem}

The Jiang--Su algebra $\cZ$ was introduced in \cite{JiangSu} and it plays a crucial role in Elliott's Classification Program. It is itself strongly self-absorbing and for any other strongly self-absorbing $C^*$-algebra $\cD$ there is a unique (up to homotopy) unital embedding $\cZ\map\cD$ \cite{WinZStable}. There is also a canonical unital $*$-homomorphism $\Oinf\map\Oinf\prot\cQ$, where $\cQ$ is the universal UHF algebra. Hence we obtain the following sequence of colocalizations of $\pNSp$ (see Equation \ref{diamond} and the comment thereafter): \beq\quad\quad \pNSp[(\Oinf\prot\cQ)^{-1}]\hookrightarrow\cdots\hookrightarrow\pNSp[(M_{2^\infty}\prot\Oinf)^{-1}]\hookrightarrow\pNSp[\Oinf^{-1}]\hookrightarrow\pNSp[\cZ^{-1}]\hookrightarrow\pNSp.\eeq There is yet another sequence of colimit preserving fully faithful functors between stable and compactly generated $\infty$-categories:

\beq \label{Loc}\cdots\hookrightarrow\pNSp[(\cpt\prot\Oinf)^{-1}]\hookrightarrow\pNSp[(\cpt\prot\cZ)^{-1}]\hookrightarrow\pNSp[\cZ^{-1}]\hookrightarrow\pNSp.\eeq

\begin{rem}
In sequence \eqref{Loc} all the stable $\infty$-categories to left of $\pNSp[\cZ^{-1}]$ can be viewed as fully faithful subcategories of $\ikk^\op$ that we are shortly going to introduce in Section \ref{KKcoloc}.
\end{rem}

\subsection{Bootstrap categories} \label{GenBoot}
Let us explain the construction of the bootstrap category in a general setting. Let $\cC$ be a compactly generated and closed symmetric monoidal stable $\infty$-category with unit object $\one_\cC$, which is a compact object. Let $V$ denote a set of compact objects of $\cC$ and let $\langle V\rangle$ denote the full stable $\infty$-subcategory of $\cC$ generated by the translations (in both directions) and cofibers of the objects in $V$. Observe that all objects of $\langle V\rangle$ are again compact in $\cC$. Therefore, $\Ind_\omega(\langle V\rangle)$ is a compactly generated full stable $\infty$-subcategory of $\cC$ (see Proposition 5.3.5.11 (1) of \cite{LurToposBook}), which is the {\em bootstrap category} in $\cC$ {\em generated by $V$}. Two different sets $V,V'$ may generate equivalent bootstrap categories. If $V=\{\one_\cC\}$ is singleton, then $\Ind_\omega(\langle\one_\cC\rangle)$ is called {\em the bootstrap category} in $\cC$; we also refer to $\Ind_\omega(\langle\one_\cC\rangle)^\op$ as {\em the bootstrap category} in $\cC^\op$. The bootstrap category $\Ind_\omega(\langle V\rangle)$ generated by a set $V$ of compact noncommutative spectra is intuitively the subcategory of noncommutative spectra that can be constructed by using the objects of $V$ as basic building blocks. 

\begin{ex}
Typical choices for $V$ constitute a small set of simple separable $C^*$-algebras.
  \begin{enumerate}
  \item The bootstrap category in $\pNSp$ generated by $\{\CC\}$ is the stable $\infty$-category of spectra $\Sp$ (see Theorem \ref{NSp}).
  \item If $V=\{ M_n(\CC) \,|\, n\in\NN\}$ then the bootstrap category in $\pNSp$ generated by $V$ is the category of {\em noncommutative stable cell complexes} \cite{EilLorPed, ThomThesis, MyNGH}. 
 \end{enumerate}
\end{ex}

\section{$\cpt$-colocalization of noncommutative spectra} \label{KKcoloc} It was remarked by the author in \cite{MyNSH} that the opposite of a stable $\infty$-categorical model for $\KK$-theory can be constructed as an accessible localization of the stable presentable $\infty$-category $\Sp(\iNS)$ (see Remark 2.29 of \cite{MyNSH}). Since $\pNSp$ is symmetric monoidal we are able to show that a smashing colocalization of noncommutative spectra furnishes us with an $\infty$-categorical incarnation of (the opposite of) $\KK$-category. An argument similar to Lemma 3.3 of \cite{MyNSHLoc} shows that $\Sigmat(\cpt)$ is a coidempotent object in $\pNSp$. It follows from Proposition 3.4 of \cite{MyNSHLoc} that the functor $R = R_\cpt :\pNSp\map \pNSp$ sending $X\mapsto X\otimes \Sigmat(\cpt)$ is a colocalization. We denote the essential image of the colocalization by $\pNSp[\cpt^{-1}]:= R (\pNSp)$, which is a closed symmetric monoidal $\infty$-category (see Proposition \ref{locCompGen}).

\begin{defn}
 We set $\ikk:= \pNSp[\cpt^{-1}]^\op$. The symmetric monoidal stable $\infty$-category $\ikk$ is our $\infty$-categorical model for the bivariant $\K$-theory category.
\end{defn}

\noindent
There is a composite functor $$\iNS\overset{\Sigmat}{\map}\pNSp\overset{R}{\map}\pNSp[\cpt^{-1}],$$ whose opposite functor is denoted by $\k:\iNS^\op\map\ikk$. 

\begin{defn} \label{Kdefn}
 For any pointed noncommutative space $X\in\iNS$ we define \beqn
 \K_*(X) &=& \pNSp[\cpt^{-1}](\k(\CC),\Sigma^*\k(X)) = \h\ikk(\Sigma^*\k(X),\k(\CC)) \\
 \K^*(X) &=& \pNSp[\cpt^{-1}](\k(X),\Sigma^*\k(\CC)) = \h\ikk(\Sigma^*\k(\CC),\k(X))
 \eeqn
where $\K_*(X)$ (resp $\K^*(X)$) is the $\K$-homology (resp. $\K$-theory) of $X$.
\end{defn}

\begin{thm} \label{klim}
The functor $\k:\iNS^\op\functor\ikk$ is symmetric monoidal and it preserves limits.
\end{thm} 

\begin{proof}
 Since $R$ is a smashing colocalization the functor $\k:\iNS^\op\functor\ikk$ is symmetric monoidal (see Example 2.2.1.7 and Proposition 2.2.1.9 of \cite{LurHigAlg}). The functor $\k$ preserves limits because its opposite functor $R\circ\Sigmat$ preserves colimits. 
\end{proof}

Now we justify that the above definitions are good ones. We continue to denote by $\k:\h\iNS^\op\map\h\ikk$ the induced functor at the level of homotopy categories. By abuse of notation we also denote the composite functor $\iCsep\overset{j^\op}{\map}\iNS^\op\overset{\k}{\map}\ikk$ as well as the functor that it induces at the level of homotopy categories by $\k$. Let $\KKcat$ denote the Kasparov bivariant $\K$-theory category for separable $C^*$-algebras, whose morphisms are given by $\KKcat(A,B) = \KK_0(A,B)$ and the composition of morphisms is induced by Kasparov product. There is a canonical functor $\Csep\map\KKcat$, which is identity on objects.

\begin{thm} \label{KK}
 For any two separable $C^*$-algebras $A,B$ there is a natural isomorphism $$\KKcat(A,B)\cong\h\ikk(\k(A),\k(B))=\h\pNSp[\cpt^{-1}]^\op(\k(A),\k(B)).$$ In other words, the functor $\k$ induces a fully faithful functor $\KKcat\map\h\ikk$.
\end{thm}

\begin{proof}
There is a fully faithful exact functor $\Theta:\SHo\map\h \pNSp^\op$ (see Proposition \ref{Sho}). There is another natural identification $\SHo(A\prot\cpt,B\prot\cpt)\cong\KKcat(A,B)$ (see Theorem 8.28 of \cite{CunMeyRos}). Using Theorem 13.7 of \cite{CunMeyRos} and the comment thereafter one deduces that there is a localization of triangulated categories $\SHo\map\KKcat\map\SHo$, where the first functor acts on objects as $A\mapsto A\prot\cpt$ and the second functor is fully faithful. Since the localization $R^\op:\h \pNSp^\op\map \ikk$ is smashing we get a commutative diagram:
\beqn
\xymatrix{
\h\iCsep\ar[r]^\iota\ar[rd] &\SHo\ar[rr]^\Theta \ar[d] && \h \pNSp^\op\ar[d]^{R^{\op}} \\
&\KKcat\ar[rr]^\theta\ar[d] && \h\ikk\ar[d] \\
&\SHo\ar[rr]^\Theta &&  \h \pNSp^\op,
}
\eeqn where the composite left vertical functor $\SHo\map\KKcat\map\SHo$ is the localization described above. The composition $R^\op\circ\Theta\circ\iota$ is the functor $\k$ and the middle horizontal functor $\theta:\KKcat\map\h\ikk$ continues to be fully faithful. Hence we get the desired natural isomorphism $$\KKcat(A,B)\cong\h\ikk(\k(A),\k(B)).$$ 
\end{proof}

Let $\iota:\CC\map\cpt$ denote the $*$-homomorphism which sends $1$ to $e_{11}$. This induces a map $\k(\iota):\k(\CC)\map\k(\cpt)$ in $\ikk$. 
\begin{cor}\label{Kstable}
 The map $\k(\iota)$ is an equivalence in $\ikk$.
\end{cor}

\begin{proof}
It is well-known that the map $\iota:\CC\map\cpt$ is a $\KK$-equivalence. By the above Theorem \ref{KK} $\k(\iota)$ must descend to an isomorphism in $\h\ikk$.
\end{proof}

\begin{rem} \label{SymMonKK}
 The bivariant $\K$-theory category for separable $C^*$-algebras $\KKcat$ is also a tensor triangulated category, where the tensor structure is induced by the maximal $C^*$-tensor product. The minimal $C^*$-tensor product would have worked equally well over here; see, for instance, \cite{Ambrogio} for some interesting features of this tensor triangulated category. Our project initiated with the noncommutative stable homotopy category $\NSH$ that was shown to be a tensor triangulated category with respect to the maximal $C^*$-tensor product (see Theorem 3.3.7 of \cite{ThomThesis}). Hence for the sake of consistency we have used the maximal $C^*$-tensor product throughout. As before one can view $\KKcat$ as a tensor triangulated subcategory of $\h\ikk$. 
\end{rem}

\subsection{Coproducts and products in $\h\ikk$} For a countable family of separable $C^*$-algebras $\{A_n\}_{n\in\NN}$ the {\em infinite sum} $C^*$-algebra in $\Csep$ is defined as $\oplus_{n\in\NN} A_n :=\dlim_{F\subset\NN} \oplus_{i\in F} A_i$, where $F$ runs through all finite subsets of $\NN$. The infinite sum $C^*$-algebra is neither the coproduct nor the product in $\Csep$. Nevertheless, when viewed inside the Kasparov category $\KKcat$ for {\em separable} $C^*$-algebras it acts as a countable coproduct, i.e., $\KKcat(\oplus_{n\in\NN} A_n, B)\cong\prod_{n\in\NN} \KKcat(A_n,B)$ for any separable $C^*$-algebra $B$ viewed as an object of $\KKcat$ (see Theorem 1.12 of \cite{RosSch}). This result is optimal, i.e., $\KKcat$ does not admit arbitrary coproducts. In sharp contrast we have 

\begin{prop} \label{limits}
The triangulated category $\h\ikk$ admits all small coproducts.
\end{prop}

\begin{proof}
Being an accessible localization of the stable presentable $\infty$-category $\Sp(\iNS)$, the stable $\infty$-category $\pNSp$ is itself presentable. It also follows from Corollary 5.5.2.4 of \cite{LurToposBook} that it admits all small limits whence the triangulated homotopy category $\h \pNSp$ admits all small products. Therefore, $\h \pNSp^\op$ admits all small coproducts. Now by construction $R:\pNSp\map \pNSp$ is a smashing colocalization, whose essential image is $\ikk^\op$. It follows that $R^\op:\h\pNSp^\op\map \h \pNSp^\op$ is a smashing localization of triangulated categories, which preserves small coproducts. Hence its image $\h\ikk$ is closed under taking small coproducts (see Remark 5.5.2 of \cite{KraLoc}).
\end{proof}

\begin{rem} \label{AlgKtheory}
The above construction of bivariant $\K$-theory is quite flexible and can be carried out purely in the algebraic setting. For instance, let $k$ be a commutative ring with unit. Consider a small full subcategory of the category of $k$-algebras with algebraic homotopy equivalences as a category with weak equivalences. One technical point is to ensure that the chosen subcategory admits finite homotopy limits. Applying the Dwyer--Kan localization one obtains a simplicial category. Taking the fibrant replacement of this simplicial category in the model structure on simplicial categories constructed in \cite{BerModel} and applying the homotopy coherent nerve to it produces an $\infty$-category. The (algebraic) homotopy equivalences between $k$-algebras become equivalences in this $\infty$-category. Now one can follow the steps as above replacing $\cpt$-stability by matrix stability that needs to be enforced by (co)localization.
\end{rem}

\subsection{Brown representability in $\h\ikk$ and its dual $\h\ikk^\op$}
Let $\cT$ be a triangulated category with arbitrary coproducts. A {\em localizing} subcategory of $\cT$ is a thick subcategory that is closed under taking small coproducts. Following \cite{KraBR} one says that $\cT$ is {\em perfectly generated} by a small set $X$ of objects of $\cT$ provided the following holds:

\begin{enumerate}
 \item There is no proper localizing subcategory of $\cT$ containing all the objects in $X$,
 \item given a countable family of morphisms $\{X_i\map Y_i\}_{i\in I}$ in $\cT$, such that the map $\cT(C,X_i)\map \cT(C,Y_i)$ is surjective for all $C\in X$ and $i\in I$, the induced map $$\cT(C,{\coprod}_i X_i)\map\cT(C,{\coprod}_i Y_i)$$ is surjective. 
\end{enumerate}

\noindent
Predictably a triangulated category $\cT$ with coproducts is called {\em perfectly cogenerated} if $\cT^\op$ is perfectly generated by some small set of objects. Finally, a triangulated category is called {\em compactly generated} if it is perfectly generated by a small set of compact objects. Recall that an object $C$ of $\cT$ is {\em compact} if $\coprod_{j\in J} \cT(C,D_j)\cong\cT(C,\coprod_{j\in J} D_j)$ for every set indexed family of objects $\{D_j\}_{j\in J}$ of $\cT$. A very intuitive and equivalent definition of compact generation is the following: a triangulated category $\cT$ admitting small coproducts is {\em compactly generated} if there is a small set $T$ of compact objects that generate $\cT$, i.e., each $X\in T$ is compact and $\cT(X,Y)=0$ for every $X\in T$ implies that $Y$ is itself $0$ (see Definition 1.7 of \cite{NeeGroDual}). For the corresponding notion in the setting of $\infty$-categories see Section 5.5.7 of \cite{LurToposBook}.

\begin{lem} \label{kkcompgen}
 The triangulated category $\h\pNSp[\cpt^{-1}]=\h\ikk^\op$ is compactly generated.
\end{lem}

\begin{proof}
 The assertion is a consequence of Proposition \ref{locCompGen}.
\end{proof}

Recall that a triangulated category is said to be {\em topological} if it is triangle equivalent to the homotopy category of a stable cofibration category (see Definition 1.4 of \cite{SchTopTri}).

\begin{prop} \label{KKtop}
 Kasparov category $\KKcat$ is topological.
\end{prop}

\begin{proof}
 The above Theorem \ref{KK} shows that $\KKcat$ is equivalent to a full triangulated subcategory of $\h\ikk=\h\pNSp[\cpt^{-1}]^\op$ via the functor $\k$. Since $\pNSp[\cpt^{-1}]$ is a stable presentable $\infty$-category (see Lemma \ref{kkcompgen}) one can establish the result following the proof of Theorem 2.27 of \cite{MyNSH}.
\end{proof}

\begin{rem}
Note that our methods actually show that both $\KKcat$ and $\KKcat^\op$ are topological. Indeed, our methods exhibit $\KKcat^\op$ naturally as a full triangulated subcategory of a stable model category whence it is topological. Since the notion of a stable model category is self-dual, one also deduces that $\KKcat$ is topological.
\end{rem}

\begin{prop}
 The triangulated category $\h\ikk$ is perfectly generated.
\end{prop}

\begin{proof}
 The above Lemma shows that $\h\ikk^\op$ is compactly generated. Thus it follows that $\h\ikk^\op$ is perfectly cogenerated (see Section 5.3 of \cite{KraLoc}) whence $\h\ikk$ is perfectly generated.
\end{proof}

\begin{thm}
 Both $\h\ikk$ and $\h\ikk^\op$ satisfy Brown representability, i.e., a functor $F:\cT^\op\map\Ab$ is cohomological and sends all coproducts in $\cT$ to products in $\Ab$ if and only if $F(-)\cong\cT(-,X)$ for some object $X\in\cT$, where $\cT=\h\ikk$ or $\h\ikk^\op$.
\end{thm}

\begin{proof}
We already observed that $\h\ikk^\op$ satisfies Brown representability due to its compact generation \cite{NeeBook}. Thanks to its perfect generation Brown representability for $\h\ikk$ follows from Theorem A of \cite{KraBR}.
\end{proof}

\begin{cor} \label{triprod}
 The triangulated category $\h\ikk$ admits all small products.
\end{cor}

\begin{proof}
We refer the readers to Remark 5.1.2 (2) of \cite{KraLoc}.
\end{proof}

\begin{rem} \label{nonsep}
As we mentioned before there is a bivariant $\K$-theory category specifically designed for nonseparable $C^*$-algebras $\kkc$ that was constructed in \cite{CunMeyRos}. Our formalism also covers the bivariant $\K$-theory of nonseparable $C^*$-algebras (or pointed noncommutative compact Hausdorff spaces). Indeed, any nonseparable $C^*$-algebra gives rise to a filtered diagram of its separable $C^*$-subalgebras. Now one can take its filtered colimit in $\iNS^\op$ after applying $j^\op:\iCsep\functor\iNS^\op$. Finally one can apply the functor $\k:\iNS^\op\functor\ikk$ to land inside $\ikk$. Our bivariant $\K$-theory for genuinely nonseparable $C^*$-algebras will in general not agree with that of \cite{CunMeyRos}, although both $\kkc$ and $\h\ikk$ contain Kasparov category for separable $C^*$-algebras $\KKcat$ as a fully faithful triangulated subcategory. Clearly $\h\ikk$ has better formal properties (see Remark 8.29 of \cite{CunMeyRos}). Moreover, in the stable $\infty$-category $\ikk$ one can compute limits and colimits, which carry more refined information than the weak (co)limits and sequential homotopy (co)limits in the triangulated category $\h\ikk$ or $\kkc$.
\end{rem}

\begin{rem}
 It would be interesting to understand the relationship between $\h\ikk$ and the bivariant $\K$-theory for $\sC$-algebras \cite{CunGenBivK} (see also \cite{MyTwist}). Note that the notion of a $\sC$-algebra is more restrictive than that of an arbitrary noncommutative (pointed) space. 
\end{rem}

\section{Bootstrap category in $\ikk$} \label{bkk}
The Universal Coefficient Theorem (UCT) is a milestone in the development of bivariant $\K$-theory and it is very natural to seek a generalization of this result beyond the category of separable $C^*$-algebras (or that of pointed noncommutative compact metrizable spaces). The original construction of the UCT class in Kasparov bivariant $\K$-theory category is due to Rosenberg--Schochet \cite{RosSch}. A separable $C^*$-algebra $A$ belongs to the UCT class if for every $B\in\Csep$ there is a natural short exact sequence of $\ZZ/2$-graded abelian groups $$0\map \Ext_*(\K_{*+1}(A),\K_{*}(B))\map\KK_*(A,B)\map \Hom_*(\K_*(A),\K_*(B))\map 0.$$ The UCT class can also be characterized as consisting of those $C^*$-algebras, which lie in the replete triangulated subcategory of $\KKcat$ generated by $\CC$ that is also closed under countable coproducts. Henceforth we set $\csus =\C_0((0,1))$ for notational clarity and it follows from Lemma 2.9 of \cite{MyNSH} that $\csus \prot (-) \cong \Sigma_{\iNS}(-)$. The following result generalizes the arguments of the proof of Bott periodicity in \cite{CunBott}. A nice observation made by the anonymous referee enabled us to formulate the result in its current form and streamline the proof.

\begin{prop}
There is an isomorphism of endofunctors $\Sigma_{\ikk}^{-2}(-) \cong \Id(-)$ of $\h\ikk$.
\end{prop}

\begin{proof}
Since $\k(\CC)$ is the tensor unit of $\h\ikk$, we have $\Sigma^{-1}_{\ikk}(-)\cong\Sigma^{-1}_{\ikk}\k(\CC)\otimes (-)$ and $\k(\CC)\otimes (-)\cong \Id(-)$. Note that $\Sigma^{-1}_{\ikk}\k(\CC) =\Omega_{\ikk} \k(\CC)\simeq \k(\Sigma_{\iNS} \CC)\simeq \k (\csus\prot \CC)\simeq \k (\csus)$. Here we have used that fact that $\csus\prot\CC \simeq \Sigma_{\iNS} \CC$ in $\iNS$ (see Lemma 2.9 of \cite{MyNSH}). Hence there is an isomorphism of endofunctors $\Sigma^{-2}_{\ikk} (-)\cong\k(\csus)\otimes\k(\csus) \otimes(-)$. Thus it suffices to show that $\k(\csus)\otimes\k(\csus) \cong \k(\CC)$ in $\h\ikk$. 

Consider the reduced Toeplitz extension $0\map\cpt\map T_0\map\csus\map 0$. Applying the functor $\k$ we obtain a diagram  $\k(\cpt)\map\k(T_0)\map\k(\csus)$ in $\ikk$. It is a (co)fiber sequence in $\ikk$ thanks to excision with respect to semisplit extenstions and it gives rise to an exact triangle in $\h\ikk$. For any $X\in\h\ikk$ applying the functor $\h\ikk(X,-)$ we get a long exact sequence 
$$\cdots\map\h\ikk(X,\k(\cpt))\map\h\ikk(X,\k(T_0))\map \h\ikk(X,\k(\csus))\map\cdots.$$  Hence we get a boundary map $$\h\ikk(X,\k(\csus)\otimes\k(\csus))\map\h\ikk(X,\k(\cpt)).$$ Now we claim that
\begin{enumerate}
 \item $\h\ikk(X,\k(T_0))=0$ and
 \item $\h\ikk(X,\k(\cpt))\cong\h\ikk(X,\k(\CC))$.
 \end{enumerate} Since the reduced Toeplitz algebra $T_0$ is $\KK$-equivalent to $0$, we get (1) from Theorem \ref{KK}. For (2) we simply invoke Corollary \ref{Kstable}. Thus we have shown that $$\h\ikk(X,\k(\csus)\otimes\k(\csus))\cong\h\ikk(X,\k(\CC))$$ for every $X\in\h\ikk$. Using the Yoneda Lemma the assertion follows.
 \end{proof}

\begin{lem}
 For any $A\in\Csep$ there are natural isomorphisms: $$\text{$\h\ikk(\k(\CC),\k(A))\cong\K_0(A)$ and $\h\ikk(\k(\csus),\k(A))\cong\K_1(A)$.}$$
\end{lem}

\begin{proof}
The assertion follows from Theorem \ref{KK} and the above Proposition.
\end{proof}

\noindent
We spell out the construction of the {\em bootstrap category} following subsection \ref{GenBoot}. Recall that there is a composite functor $$\iNS\overset{\Sigmat}{\map}\pNSp\overset{R}{\map}\pNSp[\cpt^{-1}],$$ whose opposite functor gives us $\k:\iNS^\op\map\ikk$. Let $\cC$ denote the stable $\infty$-subcategory of $\ikk^\op$ generated by $\k(\CC)$. It is the closure of $\k(\CC)$ under translations (in both directions) and cofibers. Now we set $\uct:=\Ind_\omega(\cC)^\op$, which is a stable $\infty$-subcategory of $\ikk$. Hence the homotopy category of $\Ind_\omega(\cC)$ is a localizing subcategory of $\h\ikk^\op$ compactly generated by $\k(\CC)$, since $\k(\CC)$ is compact in $\ikk^\op$. 

\begin{defn}
 We define $\uct$ (resp. $\h\uct$) to be the {\em bootstrap category} in $\ikk$ (resp. in $\h\ikk$) (see also Remark \ref{KhomUCT} and Definition \ref{BootstrapK} below).
\end{defn}

Let $\ZZ[u,u^{-1}]$ be a differential graded algebra with trivial differentials and $\deg(u)=2$ and let $D(\ZZ[u,u^{-1}])$ denote its unbounded derived category of differential graded modules. The derived category $D(\ZZ[u,u^{-1}])$ is also the homotopy category of a stable model category. Our bivariant $\K$-theory category possesses the {\em correct} formal properties from the viewpoint of homotopy theory. Thus we are able to use a result of Bousfield \cite{BouKLoc} and Franke \cite{Franke} (written up carefully in \cite{Patchkoria}; see also \cite{Roitzheim}) to arrive at an algebraic description of the triangulated category $(\h\uct)^\op$. 

\begin{thm} \label{algDesc}
 There is an additive equivalence of categories $(\h\uct)^\op\simeq D(\ZZ[u,u^{-1}])$.
\end{thm}

\begin{proof}
 Since $(\uct)^\op:=\Ind_\omega(\cC)$ is a presentable stable $\infty$-category the triangulated category $(\h\uct)^\op$ admits infinite coproducts. It follows from Theorem \ref{KK} that the graded Hom object $\Hom_{(\h\uct)^\op}(\Sigma^* \k(\CC),\k(\CC))\cong \Hom_{(\h\ikk)^\op}(\Sigma^* \k(\CC),\k(\CC))$ is isomorphic to $\ZZ[u,u^{-1}]$ with $\deg(u) = 2$. Since the graded global dimension of $\ZZ[u,u^{-1}]$ is $1$ and it is concentrated in even dimensions, one deduces the assertion from Proposition 5.2.3 of \cite{Patchkoria}.
\end{proof}

\begin{rem} \label{nonAlgebraic}
 It is not clear whether the above additive equivalence of categories is actually an exact equivalence of triangulated categories (see Remark 5. 2.4 of \cite{Patchkoria}). Hence we are not able to conclude that $(\h\uct)^\op$ is {\em algebraic} as a triangulated category according to the definition in \cite{SchAlgTri}.
\end{rem}

\noindent
Now we are going to justify the notation $\uct$ with the help of two simple propositions.

\begin{prop} \label{homuct}
Let $\cT$ denote the triangulated category $(\h\uct)^\op$. Then for any $X,Y\in\cT$ there is a natural short exact sequence $$0\map \Ext^1_{\K_*(\k(\CC))}(\K_{*}(\Sigma X),\K_*(Y))\map \cT(X,Y)\map \Hom_{\K_*(\k(\CC))}(\K_*(X),\K_*(Y))\map 0.$$ In particular, $X$ is isomorphic to $Y$ in $\cT$ if and only if $\K_*(X)$ and $\K_*(Y)$ are isomorphic as graded $\K_*(\k(\CC))\simeq\ZZ[u,u^{-1}]$-modules.
\end{prop}

\begin{proof}
The assertion is a consequence of Proposition 5.1.1 of \cite{Patchkoria}.
\end{proof}

\begin{rem} \label{KhomUCT}
 The above result is a universal coefficient theorem in $\h\ikk$ via $\K$-homology.
\end{rem}

\begin{prop}
 Let $A$ be a nuclear separable $C^*$-algebra satisfying UCT with finitely generated $\K$-theory. Then $\k(A)$ belongs to $\h\uct$. 
\end{prop}

\begin{proof}
Under the assumptions $A$ is $\KK$-equivalent to $\C(X,x)$, where $(X,x)$ is a finite pointed CW complex (see Corollary 7.5 of \cite{RosSch}). It is clear that $\k(\C(X,x))$ belongs to $\h\uct$. The assertion now follows from Theorem \ref{KK}.
\end{proof}

\subsection{The $\K$-theoretic bootstrap category} \label{Kthuct}
It follows from Theorem \ref{KK} that there is a canonical fully faithful functor $\k:\KKcat\hookrightarrow\h\ikk$. Let $\cB_{C^*}$ denote the triangulated subcategory of $\KKcat$ consisting of those separable $C^*$-algebras that satisfy UCT. Let $\K_*(-)$ denote the $\ZZ/2$-graded $\K$-theory functor on $\KKcat$. One interpretation of $\cB_{C^*}$ is that it is the the Verdier quotient $\KKcat/\ker(\K_*)$ (see Theorem 13.11 of \cite{CunMeyRos}).

\noindent
We know from Proposition \ref{limits} that the triangulated category $\h\ikk$ admits arbitrary coproducts. We denote the coproduct in $\h\ikk$ by $\coprod$. Let $\Ab^{\ZZ/2}$ denote the category of $\ZZ/2$-graded abelian groups. The object $\k(\CC)\coprod\k(\csus)$ corepresents a functor $\h\ikk\functor\Ab^{\ZZ/2}$ that generalizes the functor $\K_*:\KKcat\functor\Ab^{\ZZ/2}$, i.e., for any separable $C^*$-algebra $A$ one has \beqn
\h\ikk(\k(\CC)\coprod\k(\csus),\k(A))&\cong& \h\ikk(\k(\CC),\k(A))\oplus \h\ikk(\k(\csus),\k(A))\\ &\cong& \K_0(A)\oplus\K_1(A) = \K_*(A).\eeqn

\noindent
Let us denote the corepresented functor $\h\ikk(\k(\CC)\coprod\k(\csus),-):\h\ikk\functor\Ab^{\ZZ/2}$ by $\bK$. We have the following commutative diagram: \beqn
\xymatrix{ 
\KKcat \ar[r]^{\k}\ar[d]_{\K_*} & \h\ikk\ar[d]^{\bK} \\
\Ab^{\ZZ/2} \ar[r] & \Ab^{\ZZ/2}.
}\eeqn Let $\cN$ denote the triangulated subcategory of $\h\ikk$ spanned by the objects in the image of $\k(\ker(\K_*))$. Since $\cN^\op$ is contained in the compact objects of $(\h\ikk)^\op$, one concludes that the localizing subcategory $\langle\langle{\cN}^\op\rangle\rangle$ of $(\h\ikk)^\op$ generated by $\cN^\op$ is compactly generated. It follows that there is a coproduct preserving (Bousfield) localization of triangulated categories $(\h\ikk)^\op\functor (\h\ikk)^\op/\langle\langle{\cN}^\op\rangle\rangle$ and hence the product preserving (Bousfield) colocalization $\h\ikk\functor \cB$, where $\cB$ is the opposite of the triangulated category $(\h\ikk)^\op/\langle\langle{\cN}^\op\rangle\rangle$. 

\begin{defn} \label{BootstrapK}
 We define the triangulated category $\cB$ to be the {\em $\K$-theoretic bootstrap category}. A justification for this nomenclature will be provided below (see Theorem \ref{Kbstrap}).
\end{defn}

\begin{rem}
 It follows from Theorem 7.2.1 (2) of \cite{KraLoc} that the triangulated category $\cB^\op$ is compactly generated.
\end{rem}

\begin{prop} \label{bstrap}
 There is a fully faithful exact functor $\cB_{C^*}\functor\cB$.
\end{prop}

\begin{proof}
 The proof follows from standard arguments along the lines of Theorem 2.26 of \cite{MyNSH}.
\end{proof}

\begin{thm} \label{Kbstrap}
 Set $\pi_*(-) = \h\ikk(\Sigma^*\k(\CC),-)$ and let $X\in\cB$. Then for any $Y\in\h\ikk$ there is a natural short exact sequence $$0\map\Ext^1(\pi_{*+1}(X),\pi_*(Y))\map\h\ikk(X,Y)\map\Hom(\pi_*(X),\pi_*(Y))\map 0.$$
\end{thm}

\begin{proof}
 Since the triangulated category $\h\ikk$ admits coproducts (see Proposition \ref{limits}) the argument in the proof of Theorem 13.11 (and Exercise 13.13) of \cite{CunMeyRos} goes through. 
\end{proof}

\begin{rem}
 Observe that $\pi_*(\k(X))=\K^*(X)$, where $\K^*(X)$ is the $\K$-theory of $X\in\iNS$ (see Definition \ref{Kdefn}). Thus the objects of $\cB$ satisfy a $\K$-theoretic universal coefficient theorem. The ad hoc notation $\pi_*(-)$ in Theorem \ref{Kbstrap} is deployed to make the short exact sequence resemble the usual UCT sequence in $\KK$-theory.
\end{rem}

\begin{rem}
Let $\ker(\bK)$ denote the full triangulated subcategory of $\h\ikk$ consisting of those objects $X$ with $\bK(X)\simeq 0$. The triangulated category $\h\ikk$ admits arbitrary products (see Corollary \ref{triprod}) and the subcategory $\ker(\bK)\subseteq\h\ikk$ is colocalizing. However, it is not clear whether $\ker(\bK)^\op$ is compactly generated (this is related to the smashing conjecture in $\h\pNSp[\cpt^{-1}]$ that is an interesting problem in its own right). Hence the predictable analogue of Proposition \ref{bstrap} may not hold in this case.
\end{rem}

\section{$\cZ$-colocalization of noncommutative spectra} \label{Zcoloc}
Here we use the basic terminology of rings and modules in the context of $\infty$-categories from Sections 3 and 4 of \cite{LurHigAlg}. Recall from subsection \ref{ColocNSp} that there is a smashing colocalization $R_\cZ:\pNSp\map\pNSp[\cZ^{-1}]$. We set $\izz=\pNSp[\cZ^{-1}]^\op$ so that there is a composite functor $$\iNS\overset{\Sigmat}{\map}\pNSp\overset{R_\cZ}{\map}\pNSp[\cZ^{-1}],$$ whose opposite functor is denoted by $\z:\iNS^\op\map\izz$. Recall from Proposition \ref{Sho} that there is a fully faithful functor $\Theta:\SHo\map\h\pNSp^\op$, which maps $A$ viewed as $(A,0)\in\SHo$ to $\Sigmat(A)$. From Remark \ref{CAlg} we deduce that $\z(\CC)\simeq\Theta(\cZ)$ is a commutative algebra object in $\pNSp^\op$. Observe that $\z(\CC)=\Sigmat(\cZ)$ as an object in $\pNSp^\op$.

\subsection{Bootstrap category in $\izz$} Let $A$ be a separable $C^*$-algebra, such that $\Sigmat(A)$ is a coidempotent object in $\pNSp$. Then we have seen that $\cC=\pNSp[A^{-1}]$ is a compactly generated and closed symmetric monoidal stable $\infty$-category, whose unit object $\one_\cC=\Sigmat(A)$ is compact. Applying the construction from subsection \ref{GenBoot} we obtain the bootstrap category in $\pNSp[A^{-1}]$. Let us now investigate the bootstrap category of $\izz$. 

\begin{defn}
Specializing to the case $A=\cZ$ produces a compactly generated stable $\infty$-subcategory $\Ind_\omega(\langle\Sigmat(\cZ)\rangle)$ of $\pNSp[\cZ^{-1}]$. The subcategory $\czb :=\Ind_\omega(\langle\Sigmat(\cZ)\rangle)^\op$ of $\izz=\pNSp[\cZ^{-1}]^\op$ is said to be the {\em bootstrap category} inside $\izz$.
\end{defn}

\begin{rem}
 By construction $\h\czb$ is the colocalizing subcategory of $\h\izz$ generated by $\Sigmat(\cZ)$, since $\h\Ind_\omega(\langle\Sigmat(\cZ)\rangle)$ is the localizing subcategory of $\h\pNSp[\cZ^{-1}]$ generated by $\Sigmat(\cZ)$, i.e., the closure in $\h\pNSp[\cZ^{-1}]$ under translations, cofibers, and arbitrary coproducts.
\end{rem}

The first step towards understanding $\izz$ is to investigate the bootstrap category $\czb$. We get a description of its opposite category as a {\em module category} using the classification of stable model categories by Schwede--Shipley \cite{SchShi}. More precisely, using the enhanced version in the symmetric monoidal setup (see Proposition 7.1.2.7 of \cite{LurHigAlg}) we get 

\begin{prop}
Let $\cC=(\czb)^\op$ so that $\one_\cC=\Sigmat(\cZ)$. Then there is an equivalence $\cC\simeq \Mod_R$, where $R = \End_\cC(\one_\cC)$ is an $\EE_\infty$-ring. 
\end{prop}

\begin{proof}
By construction $\cC$ is generated by the tensor unit $\one_\cC$, which is a compact object. The tensor product preserves finite colimits separately in each variable in the stable $\infty$-category $\langle\one_\cC\rangle$. Hence it preserves small colimits separately in each variable in $\Ind_\omega(\langle\one_\cC\rangle)=\cC$ (see Corollary 4.8.1.13 of \cite{LurHigAlg}). The assertion now follows from Proposition 7.1.2.7 of \cite{LurHigAlg}.
\end{proof}

\begin{rem} \label{zalgebraization}
Our construction of the bootstrap category $\czb$ (or ${\czb}^\op$) enables us to give an algebraic recipe to compute its morphism groups. More precisely, let $\cC=(\czb)^\op$ and set $\bZ(X)=\cC(\one_\cC,X)$ (the mapping spectrum) for every $X\in\cC$. Note that $\bZ(X)$ is a module spectrum over $R=\End_\cC(\one_\cC)$. Then for any $X_1,X_2\in\cC$ there is an equivalence of spectra $\cC(X_1, X_2)\overset{\sim}{\map}\Mod_R (\bZ(X_1),\bZ(X_2))$ and there is a convergent spectral sequence (see, for instance, Corollary 4.15 of \cite{BluGepTab2}) $$\textup{E}^{p,q}_2 = \Ext^{p,q}_{\pi_{-*}(R)}(\pi_{-*}(\bZ(X_1)),\pi_{-*}(\bZ(X_2)))\Rightarrow \pi_{-p-q}\Mod_R(\bZ(X_1),\bZ(X_2)).$$
\end{rem}

We now introduce an equivalence relation on $C^*$-algebras. Let $A\in\Csep$, such that $\Sigmat(A)$ is a coidempotent object in $\pNSp$. As explained before one can consider the bootstrap category generated by a set of compact objects $V$ in $\pNSp[A^{-1}]$.

\begin{defn} \label{Ceq}
Let $B,C$ be two separable $C^*$-algebras. Then 
\begin{enumerate}
\item we say that $B,C$ are {\em $A$-equivalent} (denoted $B\sim_A C$) if $\{R_A\circ\Sigmat(B)\}$ and $\{R_A\circ\Sigmat(C)\}$ generate equivalent bootstrap categories in $\pNSp[A^{-1}]$ and

\item we say $B \leqslant_A C$ if the bootstrap category generated by $\{R_A\circ\Sigmat(B)\}$ is contained in the bootstrap category generated by $\{R_A\circ\Sigmat(C)\}$.
\end{enumerate}
\end{defn}

\begin{rem}
It is clear that $\sim_A$ is reflexive, symmetric, and transitive, i.e., an equivalence relation and $\leqslant_A$ is reflexive, anti-symmetric, and transitive, i.e., a partial order. For $A,B\in\Csep$ it is an interesting problem to determine whether $A\sim_\CC B$, i.e., whether $\{\Sigmat(A)\}$ and $\{\Sigmat(B)\}$ generate equivalent bootstrap categories in $\pNSp$.
\end{rem}

\begin{rem}
 It is conceivable that a finer equivalence relation than the one in Definition \ref{Ceq} is more interesting. One can define $B,C$ to be {\em thick $A$-equivalent} (denoted $B\sim^t_A C$) if $\{R_A\circ\Sigmat(B)\}$ and $\{R_A\circ\Sigmat(C)\}$ generate equivalent thick subcategories in $\h\pNSp[A^{-1}]$.
\end{rem}

For an $\EE_1$-ring (or an $\textup{A}_\infty$-ring spectrum) $R$ one denotes the stable $\infty$-category of right $R$-modules by $\RMod_R$ (see Chapter 7 of \cite{LurHigAlg}). Much like classical Morita theory in algebra, modulo technicalities, Morita theory in stable homotopy theory tries to ascertain when two $\EE_1$-rings have equivalent module categories. One key result in this direction is the classification of stable model categories (under some hypotheses) by Schwede--Shipley \cite{SchShi}.

\begin{prop}
If $A\sim_\CC B$, then $\RMod_{\End(\Sigmat(A))}\simeq\RMod_{\End(\Sigmat(A))}$.
\end{prop}

\begin{proof}
 Observe first that $\Ind_\omega(\langle \Sigmat(A)\rangle)$ and $\Ind_\omega(\langle \Sigmat(B)\rangle)$ are both stable and presentable $\infty$-categories that are compactly (graded) generated by $\Sigmat(A)$ and $\Sigmat(B)$ respectively. Hence by the Schwede--Shipley classification there are equivalences $$\text{$\Ind_\omega(\langle \Sigmat(A)\rangle)\simeq \RMod_{\End(\Sigmat(A))}$ and $\Ind_\omega(\langle \Sigmat(B)\rangle)\simeq \RMod_{\End(\Sigmat(B))}$}$$ (see also Theorem 7.1.2.1 of \cite{LurHigAlg} for the version needed here). Now $A\sim_\CC B$ implies by definition that $\Ind_\omega(\langle \Sigmat(A)\rangle)\simeq\Ind_\omega(\langle \Sigmat(B)\rangle)$.
\end{proof}

\begin{rem}
Algebraic $\K$-theory and topological Hochschild homology are both Morita invariant functors. For an $\EE_1$-ring $R$ the algebraic $\K$-theory functor takes as input the stable $\infty$-category of compact $R$-modules. Such functors can be deployed to test whether $A\sim_\CC B$.
\end{rem}

Let $p,q$ be relatively prime infinite supernatural numbers and $\Z_{p,q}$ be a prime dimension drop $C^*$-algebra. Let $\cQ$ be the universal UHF algebra of infinite type.

\begin{thm} \label{ColocEx}
In $\cZ$-colocalized noncommutative spectra $\pNSp[\cZ^{-1}]$ the following hold:
 \begin{enumerate}
 \item $\CC\sim_\cZ \cZ$
  \item $\CC\leqslant_\cZ \Z_{p,q}$
  \item $\cpt\nsim_\cZ \cpt\prot\cQ$
 \end{enumerate}

\end{thm}

\begin{proof}
Observe that $R_A\circ\Sigmat(B)\simeq\Sigmat(B\prot A)$. Thus for (1) we need to show that the bootstrap categories generated by $\{\Sigmat(\cZ)\}$ and $\{\Sigmat(\cZ\prot\cZ)\}$ are equivalent. Since $\cZ$ is strongly self-absorbing the unital embedding $\id\otimes\one_\cZ: \cZ\map\cZ\prot\cZ$ is a unital $*$-homomorphism between strongly self-absorbing $C^*$-algebras that is homotopic to an isomorphism. Hence $\Sigmat(\cZ)\overset{\sim}{\leftarrow}\Sigmat(\cZ\prot\cZ)$ in $\pNSp[\cZ^{-1}]$, from which the assertion follows.

Thanks to (1) for (2) we simply need to show $\cZ\leqslant_\cZ \Z_{p,q}$, i.e., we need to show that the bootstrap category generated by $\{\Sigmat(\cZ\prot\cZ)\}$ is contained in that of $\{\Sigmat(\Z_{p,q}\prot\cZ)\}$. Recall that it is shown in Proposition 3.5 of \cite{RorWin} that there are unital embeddings $\cZ\map\Z_{p,q}$ and $\Z_{p,q}\map\cZ$, so that the composition is a unital endomorphism $\cZ\map\cZ$. Since the space of unital endomorphisms of $\cZ$ is contractible (see Theorem 2.3 of \cite{DadPen}) the above composition is homotopic to the identity. By tensoring with $\cZ$ and applying $\Sigmat(-)$ we find that the composition $\Sigmat(\cZ\prot\cZ)\map\Sigmat(\Z_{p,q}\prot\cZ)\map\Sigmat(\cZ\prot\cZ)$ is homotopic to the identity in $\pNSp[\cZ^{-1}]$. Hence we conclude that $\Sigmat(\cZ\prot\cZ)$ is a retract of $\Sigmat(\Z_{p,q}\prot\cZ)$ in $\pNSp[\cZ^{-1}]$. Since the bootstrap category generated by $\{\Sigmat(\Z_{p,q}\prot\cZ)\}$ is a localizing subcategory of $\pNSp[\cZ^{-1}]$, it must be closed under retracts, i.e., $\Sigmat(\cZ\prot\cZ)$ belongs to this bootstrap category. Now the assertion follows from Proposition 5.3.5.11 (1) of \cite{LurToposBook}.

We prove (3) by contradiction. Set $E_1 = \End(\Sigmat(\cpt))$ and $E_2 =\End(\Sigmat(\cpt\prot\cQ))$ and assume $\cpt\sim_\cZ \cpt\prot\cQ$. This implies that the topological Hochschild homology $E_1$ and $E_2$ are equivalent. However, this is false.
\end{proof}

\subsection{Bousfield equivalence} An important concept in the global structure of the stable homotopy category is {\em Bousfield equivalence} \cite{BouEquiv}. It is possible to consider a variant of it for noncommutative spectra. We present a definition that is slightly different from the routine generalization of Bousfield equivalence for spectra. Indeed, unlike localizations with respect to arbitrary spectra we focus on smashing colocalizations with respect to certain $C^*$-algebras (their stabilizations are compact objects in $\pNSp$).

\begin{defn} \label{BouEq}
 Let $A,B\in\iCsep^\op$, such that $\Sigmat(A),\Sigmat(B)$ are coidempotent objects in $\pNSp$ (see subsection \ref{ColocNSp}). Then 
 \begin{enumerate}
  \item we say that $A$ and $B$ are {\em Bousfield equivalent} if $\pNSp[A^{-1}] \simeq\pNSp[B^{-1}]$ (as stable $\infty$-categories) and
  
  \item we say that $A\preceq B$ if there is a fully faithful functor $\pNSp[A^{-1}] \hookrightarrow\pNSp[B^{-1}]$.
 \end{enumerate}
\end{defn}

\begin{ex}
It follows from Theorem \ref{KK} and $\cZ$-stability in bivariant $\K$-theory that the canonical unital $*$-homomorphism $\CC\map\cZ$ induces an equivalence in $\ikk^\op =\pNSp[\cpt^{-1}]$. Therefore, the inclusion $\pNSp[(\cZ\prot\cpt)^{-1}]\hookrightarrow\pNSp[\cpt^{-1}]$ (see Example \ref{colocEx}) is an equivalence of stable $\infty$-categories whence $\cZ\prot\cpt$ and $\cpt$ are Bousfield equivalent.
\end{ex}

\begin{ex}
If $\cD\map\cD'$ is an injective $*$-homomorphism between strongly self-absorbing $C^*$-algebras, then it follows from Proposition \ref{SScoloc} and Example \ref{SScolocEx} that $\cD'\preceq\cD$. 
\end{ex}

\begin{rem}
In view of the results in \cite{SchExotic} it is important to work with stable $\infty$-categories rather than their triangulated homotopy categories in the above definition. Note that the definition extends effortlessly to all coidempotent objects in $\pNSp$ (and eventually to all objects in $\pNSp$ with some effort). It would be interesting to classify strongly self-absorbing $C^*$-algebras up to Bousfield equivalence (according to Definition \ref{BouEq}).
\end{rem}


\bibliographystyle{abbrv}
\bibliography{/home/mahanta/Professional/math/MasterBib/bibliography}

\vspace{5mm}

\end{document}